\documentclass[11pt]{article}
\usepackage{amsmath,amsthm,amssymb,amsfonts}
\usepackage{latexsym}
\usepackage{graphicx,psfrag,import}
\usepackage{fullpage}
\usepackage{framed}
\usepackage{verbatim}
\usepackage{color}
\usepackage{epsfig}
\usepackage{epstopdf}
\usepackage{hyperref}
\usepackage{geometry}
\usepackage{mathtools}
\usepackage{nonfloat}
\usepackage{enumerate}
\usepackage{multicol}
\usepackage{a4wide}
\usepackage{booktabs}
\usepackage{enumitem}
\usepackage{lineno}
\usepackage{parcolumns}
\usepackage{thmtools}
\usepackage{xr}
\usepackage{epstopdf}
\usepackage{mathrsfs}
\usepackage{subfig}
\usepackage{caption}
\usepackage{comment}
\usepackage{authblk}
\usepackage{setspace}

\theoremstyle{plain}
\newtheorem{theorem}{Theorem}[section]
\newtheorem{corollary}[theorem]{Corollary}

\theoremstyle{definition}
\newtheorem{definition}[theorem]{Definition}

\newtheorem{remark}[theorem]{Remark}

\theoremstyle{remark}

        {\begin{list}
                {\noindent\makebox[0mm][r]{$\bullet$}}
                {\leftmargin=5.5ex \usecounter{enumi}
      \topsep=1.5mm \itemsep=-.75ex}
        }
        {\end{list}}

\newcommand\f{\boldsymbol{f}}   
\newcommand\bk{\boldsymbol{k}}   
\newcommand\x{\boldsymbol{x}}
\newcommand\y{\boldsymbol{y}}

\newcommand\by{\boldsymbol{y}}

\newcommand\GG{\mathcal{G}}

\begin{document}

\title{Homeostasis and injectivity: a reaction network perspective}

\author[1]{Gheorghe Craciun}
\author[2]{Abhishek Deshpande}
\affil[2]{Department of Mathematics and Department of Biomolecular Chemistry, University of Wisconsin-Madison, {\tt craciun@math.wisc.edu}.}
\affil[3]{Department of Mathematics, University of Wisconsin-Madison, {\tt deshpande8@wisc.edu}.}

\maketitle

\begin{abstract}
\noindent
{\em Homeostasis} represents the idea that a  feature may remain invariant despite  changes in some external parameters. 
We establish a connection between homeostasis and \emph{injectivity} for reaction network models. In particular, we show that a reaction network cannot exhibit homeostasis if a modified version of the network (which we call  \emph{homeostasis-associated  network}) is injective. We provide examples of reaction networks which can or cannot exhibit homeostasis by analyzing the injectivity of their homeostasis-associated networks.
\end{abstract}

\section{Introduction}

Coined by Cannon~\cite{cannon1926physiological} in 1926, the idea of \emph{homeostasis} has its roots in the work of Claude Bernard~\cite{bernard1898introduction}, and refers to a regulatory mechanism by which a feature maintains a steady state that is not perturbed by changes in the environment. Often, homeostasis involves the use of negative feedback loops that help restore a feature to its steady state. At the scale of a whole organism, homeostasis manifests itself in many forms; some prominent examples include the maintenance of body temperature, blood sugar level, concentration of ions in body fluids with changes in the external environment. Homeostasis is also exhibited in intracellular metabolism, where certain concentration remain almost unperturbed with change in concentrations of amino acids~\cite{reed2017analysis}.

As noted in~\cite{reed2017analysis}, homeostasis does not imply that the whole system remains invariant with change in external variables. In fact, changes in external variables can cause changes in certain internal variables, while other internal variables remain almost unchanged. In recent years there has been a lot of interest in identifying and analyzing homeostasis in mathematical models of biological interaction systems~\cite{golubitsky2017homeostasis, nijhout2014homeostasis, reed2017analysis, nijhout2014escape,golubitsky2018homeostasis,tang2016design}. Some of this renewed interest started with the mathematical analysis of homeostasis in the context of the folate and methionine metabolism~\cite{nijhout2004mathematical,reed2017analysis}.


While there is no universally accepted mathematical definition of homeostasis, here we focus mostly on the notion of \emph{infinitesimal homeostasis} for input-output systems, as introduced by Golubitsky and Stewart in~\cite{golubitsky2017homeostasis} and further refined in~\cite{wang2020structure, golubitsky2020input}. Our main interest is to analyze  homeostasis from the  point of view of {\em reaction network theory}~\cite{feinberg1979lectures,feinberg2019foundations,yu2018mathematical}. 

This paper is organized as follows. In Section~\ref{sec:reaction_networks}, we define reaction networks and related notions, including  the notion of {\em injective reaction network}. In Section~\ref{sec:local_homeostasis}, we introduce homeostasis as the capacity of a feature to be robust to change in the parameters of the system. We present a procedure for checking whether a reaction network may admit homeostasis by constructing a modified network and checking if it is injective (see~Theorem~\ref{thm:reaction_network_homeostasis}). 
We also describe a sufficient condition for {\em perfect~homeostasis} (see~Theorem~\ref{thm:reaction_network_perfect_homeostasis}). 
In Section~\ref{sec:examples} we present several examples of reaction networks for which their capacity to exhibit homeostasis can be analyzed using the procedure described in  Section~\ref{sec:local_homeostasis}.

\section{Euclidean embedded graphs and reaction networks}\label{sec:reaction_networks}

An \emph{Euclidean embedded graph} is a directed graph $\GG = (V,E)$, where $V\subset\mathbb{R}^n$ and $E$ are the sets of vertices and edges respectively. Associated with every edge $(\by,\by')\in E$ is a source vertex $\by\in V$ and a target vertex $\by'\in V$. An edge $(\by,\by')\in E$ will also be denoted by $\by\rightarrow\by'\in E$.

A \emph{reaction network} is a Euclidean embedded graph $\GG = (V,E)$, where $V\subset \mathbb{R}^n_{\ge 0}$ and $E$ is the set of edges that correspond to reactions in the network~\cite{craciun2015toric,craciun2019polynomial}. An alternative way of describing a reaction network is by specifying a set of species and a  set of reactions. For example, consider the set of species  $\{X_1,X_2\}$ and the set of  reactions $\{2X_1\rightarrow 3X_2, \  X_1+X_2\rightarrow 3X_1\}$. The corresponding  Euclidean embedded graph lies in $\mathbb{R}^2_{\ge 0}$ and has two edges: one edge from $(2,0)$ to $(0,3)$ and one edge from $(1,1)$ to $(3,0)$, where the vertex  vectors are formed by the coefficients of the species $X_1$ and $X_2$ on the reactant side and product side, respectively.

The {\emph stoichiometric subspace} of a reaction network is the linear subspace given by ${\rm span}\{\by'-\by\, |\, \by\rightarrow \by' \in E\}$. Given a point $\x_0\in\mathbb{R}^n_{>0}$, the positive stoichiometric compatibility class of $\x_0$ is the affine subspace $(\x_0 + S)\cap\mathbb{R}^n_{>0}$. 

There exist many choices for modelling the kinetics of reaction networks. The most common one is based on the \emph{law of mass-action}~\cite{voit2015150,guldberg1864studies,yu2018mathematical,gunawardena2003chemical,feinberg1979lectures} where, associated with each reaction $\by\rightarrow\by'$ there is a {\em rate constant} $k_{\by\rightarrow\by'}>0$, and the dynamics of the network is given by
\begin{eqnarray}\label{eq:mass_action_dynamics}
\frac{d\x}{dt}=\displaystyle\sum_{\by \to \by'\in E} k_{\by\rightarrow \by'}\x^{\by}(\by'-\by),
\end{eqnarray}
where $\x\in\mathbb{R}^n_{>0}$ and $\x^{\y}=x_1^{y_1}x_2^{y_2}\cdots x_n^{y_n}$. Let $\bk=(k_{\by\rightarrow\by'})_{\by\rightarrow\by'\in E}$ denote the vector of rate constants and $\f(\x,\bk)$ denote the right-hand side of Equation~(\ref{eq:mass_action_dynamics}). The Jacobian corresponding to the dynamical system~(\ref{eq:mass_action_dynamics}) is given by the matrix
$$J(\x,\bk)=
\begin{pmatrix} 
\frac{\partial f_1(\x,\bk)}{\partial x_1} &  \frac{\partial f_1(\x,\bk)}{\partial x_2} & \cdots &  \frac{\partial f_1(\x,\bk)}{\partial x_n} \\
\frac{\partial f_2(\x,\bk))}{\partial x_1} &  \frac{\partial f_2(\x,\bk)}{\partial x_2} & \cdots &  \frac{\partial f_2(\x,\bk)}{\partial x_n} \\
\vdots  & \vdots   \\
\frac{\partial f_n(\x,\bk)}{\partial x_1} &  \frac{\partial f_2(\x,\bk)}{\partial x_2} & \cdots &  \frac{\partial f_n(\x,\bk)}{\partial x_n} 
\end{pmatrix},
$$
where $\x=(x_1,x_2,...,x_n)$ and $f_i(\x,\bk)$ is the right-hand side of the dynamics corresponding to species $X_i$. A point $\x_0\in\mathbb{R}^n_{>0}$ is said to be an {\em equilibrium} of~(\ref{eq:mass_action_dynamics}) if $$\displaystyle\sum_{\by \to \by'\in E} k_{\by\rightarrow \by'}\x_0^{\by}(\by'-\by)=0.$$ 
An equilibrium point $\x_0\in\mathbb{R}^n_{>0}$ of~(\ref{eq:mass_action_dynamics}) is said to be \emph{complex balanced} if for every vertex $\by$ in the reaction network we have
\begin{eqnarray}
\displaystyle\sum_{\by \to \by'\in E} k_{\by\rightarrow \by'}\x_0^{\by} = \displaystyle\sum_{\by' \to \by\in E} k_{\by\rightarrow \by'}\x_0^{\by'}.
\end{eqnarray}
An equilibrium $\x_0\in\mathbb{R}^n_{>0}$ is said to be a \emph{linearly stable equilibrium} of~(\ref{eq:mass_action_dynamics}) if the eigenvalues of the Jacobian of~(\ref{eq:mass_action_dynamics}) evaluated at the point $\x_0$ have negative real parts. It is known from~\cite[Theorem 5.2]{siegel2008linearization} and~\cite[15.2.2]{feinberg2019foundations} that complex balanced equilibria are linearly stable.

Recall the $\f(\x,\bk)$ denotes the right hand side of Equation~(\ref{eq:mass_action_dynamics}). Now consider the function $\x\rightarrow \f(\x,\bk)$. A reaction network $\GG$ is said to be injective if the function $\x\rightarrow \f(\x,\bk)$ corresponding to the dynamics given by~(\ref{eq:mass_action_dynamics}) is injective for all $\bk$. A point $\x_0\in\mathbb{R}^n_{>0}$ is said to be an equilibrium of~(\ref{eq:mass_action_dynamics}) if $\f(\x_0,\bk)=0$. It follows that an injective reaction network cannot have multiple equilibria.

In general, it is extremely difficult to determine whether the function $\x\rightarrow \f(\x,\bk)$ is injective or not. A necessary and sufficient condition for a reaction network to be injective is given by Theorem 3.1 in~\cite{craciun2005multiple}, which we state here:

\begin{theorem}\label{thm:injectivity}
Consider a reaction network $\GG$ with species given by $X_1,X_2,...,X_n$. Let $J(\x,\bk)$ denote the Jacobian corresponding to the dynamics generated by $\GG$. Then $\GG$ is injective if and only if $det(J(\x,\bk))$ is non-zero for every $\x\in\mathbb{R}^n_{>0}$ and for all choice of rate constants $\bk$. Further, by Theorem 3.2 in~\cite{craciun2005multiple}, there is a one-to-one correspondence between the coefficients in the expansion of $det(J(\x,\bk))$ and products of the form $det(\y_1,\y_2,...,\y_n)det(\y_1-\y'_1,\y_2-\y'_2,...,\y_n-\y'_n)$ for all choices of $n$ reactions in $\GG$.
\end{theorem}

\begin{corollary}\label{cor:injectivity}
Consider a reaction network $\GG$ with species $X_1,X_2,...,X_n$. Then $\GG$ is injective if and only if products of the form $det(\y_1,\y_2,...,\y_n)det(\y_1-\y'_1,\y_2-\y'_2,...,\y_n-\y'_n)$ for any choice of $n$ reactions are all of the same sign, and at least one such product is non-zero.
\end{corollary}

\begin{proof}
Follows from Theorem~\ref{thm:injectivity}.
\end{proof}

A good tool for analyzing injectivity is the directed species reaction graph (abbreviated as DSR graph) first introduced by Banaji and Craciun in~\cite{banaji2009graph}. In what follows, we describe some terminology in the context of DSR graphs. We will denote a \emph{negative edge} in the DSR graph with a \emph{dashed line}, and a \emph{positive edge} with a \emph{bold line}. Let $|C|$ denote the length of a cycle in the DSR graph. A cycle is an e-cycle if the number of positive edges has the same parity as the $\frac{|C|}{2}$. Otherwise, it is an o-cycle. A cycle is a s-cycle if $\displaystyle\prod_{i=1}^n l(e_{2i-1})= \displaystyle\prod_{i=1}^n l(e_{2i})$, where $l(e)$ denotes the stoichiometric label of the edge $e$. We will say that two cycles have an odd intersection if their orientation is compatible and every component of their intersection contains an odd number of edges. Figure~\ref{fig:DSR_enzyme_substrate} shows the DSR graph for a modified enzyme-substrate network.\\

\begin{figure}[h!]
\centering
\includegraphics[scale=0.37]{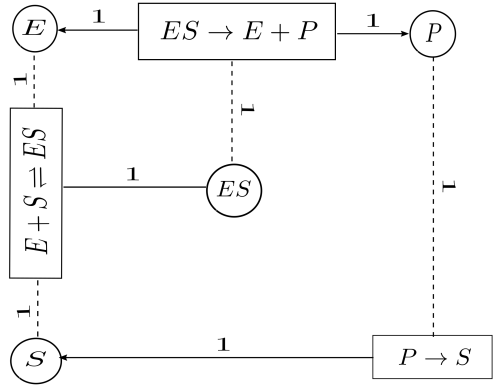}
\caption{DSR graph for a modified enzyme-substrate network given by the following reactions: $\{E + S \rightleftharpoons ES, ES \rightarrow E + P, P\rightarrow S, E\rightarrow \emptyset, P\rightarrow \emptyset\}$.}
\label{fig:DSR_enzyme_substrate}
\end{figure}

\begin{theorem}\label{thm:DSR_criterion}
\textbf{A DSR criterion}~\cite{banaji2009graph}: Consider a reaction network $\GG$. Suppose the following conditions are satisfied:
\begin{enumerate}
\item\label{s_cycle} Every e-cycle is a s-cycle in the DSR graph of $\GG$.
\item\label{e_cycle} No two e-cycles contain an odd intersection in the DSR graph of $\GG$.
\item \label{choice_reactions} there exists a choice of $n$ reactions $\{\y_1\rightarrow \y'_1,\y_2\rightarrow \y'_2,\y_3\rightarrow \y'_3,...,\y_n\rightarrow \y'_n\}$ in $\GG$ such that $det(\y_1,\y_2,\y_3,...,\y_n)det(\y'_1-\y_1,\y'_2-\y_2,\y'_3-\y_3,...,\y'_n-\y_n)\neq 0$.
\end{enumerate}
Then $\GG$ is injective.
\end{theorem}

Consider the DSR graph in Figure~\ref{fig:DSR_enzyme_substrate}. This has two cycles given by $(ES\rightarrow E+P)-(E)-(E+S\rightleftharpoons ES)-ES-(ES\rightarrow E+P)$ and $(E+S\rightleftharpoons ES)-ES-(ES\rightarrow E+P)-(P)-(P\rightarrow S)-(S)-(E+S\rightleftharpoons ES)$. Both these cycles are e-cycles and s-cycles since the number of positive edges has the same parity as half the length of the respective cycle. The intersection of these two e-cycles is the path $(E+S\rightleftharpoons ES)-ES-(ES\rightarrow E+P)$ and this is not an odd intersection since it contains two edges (which is an even number). Further, 
if we choose four reactions from the network given by the following: $\{E + S \rightarrow ES, ES \rightarrow E + P, E\rightarrow  \emptyset, P \rightarrow \emptyset\}$, then we have $det(\y_1,\y_2,\y_3,\y_4)\cdot det(\y'_1-\y_1,\y'_2-\y_2,\y'_3-\y_3,y'_4-\y_4) = det \begin{pmatrix} 
1  &  0  &  1  & 0 \\
1  &  0  &  0  & 0 \\
0  &  1  &  0  & 0 \\
0  &  0  &  0  & 1 \\
\end{pmatrix}\cdot
det\begin{pmatrix} 
-1  &  1  &  -1  & 0 \\
-1 &   0  &  0  & 0 \\
1  &  -1  &  0  & 0 \\
0  &   1  &  0  & -1 \\
\end{pmatrix}=1\cdot 1\neq 0$. Therefore, by Theorem~\ref{thm:DSR_criterion} this reaction network is injective.

\section{Homeostasis}\label{sec:local_homeostasis}

The ability of a feature to remain invariant when certain parameters of the system are changed is the essential idea behind homeostasis. A common example of homeostasis is exhibited when an organism maintains its body temperature despite fluctuations in the temperature of the environment. The temperature of the body varies linearly with temperature for low and high values of the environment temperature; however for moderate values of the environment temperature, the body temperature remains approximately constant. This variation of body temperature with the environment resembles the shape of a ``chair"~\cite{nijhout2004mathematical,nijhout2014homeostasis}. In~\cite{golubitsky2017homeostasis}, this ``chair" form provides inspiration for a definition of homeostasis in the context of singularity theory. In particular, the idea of homeostasis corresponds to the derivative of an output (homeostasis) variable with respect to an external input being zero at a certain point. As outlined in~\cite{golubitsky2017homeostasis}, we consider the following setup: Let $\x=(x_1,x_2,...,x_n)$ and consider
\begin{eqnarray}
\frac{d\x}{dt} = \mathcal{F}(\x,\zeta)
\end{eqnarray}
given by 
\begin{align}\label{sys:homeostatis}
\begin{split}
\frac{dx_1}{dt} &= f_1(x_1,x_2,...,x_n) + \zeta \\
\frac{dx_2}{dt} &= f_2(x_1,x_2,...,x_n) \\
\vdots  &= \vdots\\
\frac{dx_n}{dt} &= f_n(x_1,x_2,...,x_n) 
\end{split}
\end{align}
As in~\cite{golubitsky2017homeostasis}, \emph{throughout this paper} we assume that the variable $x_1$ is the input variable, and the output variable (which may or may not exhibit homeostatis) is $x_n$. We will also assume that there exists a linearly stable equilibrium of~(\ref{sys:homeostatis}) given by $(\x_0,\zeta_0)$. By the implicit function theorem, there exists solutions $\tilde{\x}(\zeta)$ in a neighbourhood of the equilibrium $(\x_0,\zeta_0)$ satisfying $\mathcal{F}(\tilde{\x}(\zeta),\zeta)=0$. In particular, this implies that $\tilde{\x}$ depends continuously on $\zeta$ in a neighbourhood of the equilibrium $(\x_0,\zeta_0)$. Recall the definition of homeostasis from~\cite{golubitsky2017homeostasis}: 

\begin{definition}\label{defn:homeostasis}
Consider a dynamical system of the form~(\ref{sys:homeostatis}) and let $J(\x)$ denote its Jacobian. Assume that $(\tilde{x}_1,\tilde{x}_2,...,\tilde{x}_n)$ be a linearly stable equilibrium of~(\ref{sys:homeostatis}) at $\zeta=\tilde{\zeta}$. We say that we have infinitesimal homeostasis at $(\tilde{x}_1,\tilde{x}_2,...,\tilde{x}_n,\tilde{\zeta})$ if $det(B)\biggr\rvert _{(\tilde{x}_1,\tilde{x}_2,..., \tilde{x}_n)}=0$ where $B$ is the $(n-1)\times (n-1)$ minor of $J$ obtained by deleting its first row and last column. 
\end{definition}

\bigskip

\begin{remark}
As remarked in~\cite{wang2020structure}, there exists several forms of homeostasis. Specifically, definition~\ref{defn:homeostasis} refers to \emph{infinitesimal homeostasis}, which requires the derivative of the input-output function to be zero at a point. The idea of \emph{perfect homeostasis} refers to the situation when the derivative of the input-output function vanishes on an entire interval. The notion of \emph{near perfect homeostasis} refers to the situation when the input-output function is approximately constant in a neighbourhood of a point. 
\end{remark}

\begin{definition}\label{homeostasis_association}
Consider a reaction network $\GG$. The \emph{homeostasis-associated reaction network} of $\GG$, denoted by $\tilde{\GG}$, is obtained from $\GG$ as follows
\begin{enumerate}
\item[] \textbf{Step 1}: For each reaction in $\GG$ involving the species $X_1$, modify the reaction such that stoichiometric coefficient of $X_1$ in the reactant is the same as the stoichiometric coefficient of $X_1$ in the product.

\item[] \textbf{Step 2}: Add the reaction $X_n\rightarrow X_1$. 
\end{enumerate} 
\end{definition}

\begin{theorem}\label{thm:reaction_network_homeostasis}
Consider a reaction network $\GG$ with species $X_1,X_2,...,X_n$. Let $\tilde{\GG}$ be the homeostasis-associated reaction network of $\GG$. If the  graph  $\tilde{\GG}$ satisfies the conditions~\ref{s_cycle},~\ref{e_cycle} and~\ref{choice_reactions} in Theorem~\ref{thm:DSR_criterion}, then the mass-action dynamical system generated by $\GG$ cannot exhibit infinitesimal homeostasis (with input $X_1$ and output $X_n$) for any choices of rate constants.
\end{theorem}

\begin{proof}
Let $J$ and $\tilde{J}$ denote the Jacobians coresponding to the dynamical systems generated by $\GG$ and $\tilde{\GG}$ respectively. Step 1 of the procedure in Definition~\ref{homeostasis_association} makes the first row of $\tilde{J}$ zero. Step 2 of Definition~\ref{homeostasis_association} generates a non-zero element in the top right corner of the $\tilde{J}$. Therefore, the Jacobian $\tilde{J}$ has the first row consisting entirely of zeros except the last element. In addition, $\tilde{J}$ has the same $(n-1)\times (n-1)$ minor $B$ as obtained by deleting the first row and last column of $J$. Expanding along the first row of $\tilde{J}$, we get that $det(\tilde{J})=k_{n,1}det(B)$, where $k_{n,1}$ is the rate constant corresponding to the reaction $X_n\rightarrow X_1$. By Theorem~\ref{thm:DSR_criterion}, $\tilde{\GG}$ is injective. Therefore by Theorem~\ref{thm:injectivity}, we have $det(\tilde{J}(\x))\neq 0$ for every $\x\in\mathbb{R}^n_{>0}$. This implies that $det(B(\x))\neq 0$ for every $\x\in\mathbb{R}^n_{>0}$ and hence $\GG$ cannot exhibit infinitesimal homeostasis for any choices of rate constants. 
\end{proof}

\bigskip

On the other hand, if the  graph  $\tilde{\GG}$ {\em fails} to satisfy condition~\ref{choice_reactions} in Theorem~\ref{thm:DSR_criterion}, then  we have $det(\tilde{J}(\x)) = 0$ for every $\x\in\mathbb{R}^n_{>0}$, which implies that $det(B(\x))=0$ for every $\x\in\mathbb{R}^n_{>0}$. Therefore, we obtain:

\bigskip

\begin{theorem}\label{thm:reaction_network_perfect_homeostasis}
Consider a reaction network $\GG$ with species $X_1,X_2,...,X_n$. Let $\tilde{\GG}$ be the homeostasis-associated reaction network of $\GG$. If the  graph  $\tilde{\GG}$ does not satisfy  condition~\ref{choice_reactions} in Theorem~\ref{thm:DSR_criterion}, then any mass-action dynamical system generated by $\GG$ must exhibit perfect homeostasis (with input $X_1$ and output $X_n$) at any linearly stable equilibrium.
\end{theorem}

\bigskip

In particular, note that if a network satisfies  condition~\ref{choice_reactions} in Theorem~\ref{thm:DSR_criterion}, then  the dimension of its stoichiometric subspace is $n$, which implies the following:

\begin{corollary}\label{cor:reaction_network_perfect_homeostasis}
Consider a reaction network $\GG$ with species $X_1,X_2,...,X_n$. Let $\tilde{\GG}$ be the homeostasis-associated reaction network of $\GG$. If the dimension of the stoichiometric subspace of $\tilde{\GG}$ is less than $n$, then any mass-action dynamical system generated by $\GG$ must exhibit perfect homeostasis (with input $X_1$ and output $X_n$) at any linearly stable equilibrium.
\end{corollary}

\bigskip

\begin{remark}
Recall that the notion of infinitesimal homeostasis (as described by Definition~3.1) assumes the existence of a linearly stable equilibrium $(\tilde{x}_1,\tilde{x}_2,...,\tilde{x}_n)$. For a mass-action system generated by a reaction network  $\GG$ this implicitly says that the dimension of the stoichiometric subspace of $\GG$ must be $n$. In other words, the notion of infinitesimal homeostasis (as described by Definition~3.1) {\em cannot} ever apply to a mass-action system that has one or more linear conservation laws (i.e., for which the dimension of the stoichiometric subspace of $\GG$ is less than $n$). 
\end{remark}

\section{Examples}\label{sec:examples}

The goal of this section is to demonstrate examples of reaction networks that can or cannot exhibit infinitesimal homeostasis using the procedure outlined in Definition~\ref{homeostasis_association}. 

\subsection{A reaction network that does not exhibit infinitesimal homeostasis for any choice of network parameters}

The biological motivation for the following example comes from ``sequestration networks" as defined in~\cite{joshi2015survey,felix2016analyzing}. In particular, they find instances of such networks in the transcription mechanism of E.coli. The trp operon contains genes that encode for the amino acid tryptophan. The operon is turned ``off" or ``on" depending upon the levels of tryptophan. When the tryptophan levels are low, it is turned ``off" and when the levels of tryptophan are high, it is turned ``on". In the presence of tryptophan, the trp repressor can bind to the operon sites and prevent the expression of the operon. This can be seen as a sequestration reaction $X_1 + X_2 \rightarrow \emptyset$, where $X_1$ is the tryptophan and $X_2$ is the trp operon. Taking our cue from this, we consider the following sequestration network.\\
 
\textbf{Example:} Consider the reaction network $\GG_1$ given by:

\begin{eqnarray}
\begin{split}
X_1 + X_2 &\rightarrow \emptyset \\
X_2 + X_3 &\rightarrow \emptyset \\
X_3 + X_4 &\rightarrow  \emptyset\\
X_4 &\rightarrow X_1 \\
\emptyset &\xrightleftharpoons[]{\zeta} X_1 \\
\emptyset &\rightleftharpoons X_2 \\
\emptyset &\rightleftharpoons X_3 \\
\emptyset &\rightleftharpoons X_4 \\
\end{split}
\end{eqnarray}
The homeostasis-associated reaction network corresponding to $\mathcal{G}_1$ will be denoted by $\tilde{\mathcal{G}}_1$ and is given by
\begin{eqnarray}
\begin{split}
X_1 + X_2 &\rightarrow X_1 \\
X_2 + X_3 &\rightarrow \emptyset \\
X_3 + X_4 &\rightarrow  \emptyset\\
X_4 &\rightarrow X_1 \\
\emptyset &\rightleftharpoons X_2 \\
\emptyset &\rightleftharpoons X_3 \\
\emptyset &\rightleftharpoons X_4 \\
\end{split}
\end{eqnarray}
Note that when we apply Step 1 of the procedure listed in Definition~\ref{defn:homeostasis} to the reaction $X_4 \rightarrow X_1$, we get the reaction $X_4 \rightarrow \emptyset$. Step 2 of the procedure then adds the reaction $X_4 \rightarrow X_1$. As a consequence, we get the homeostasis-associated network $\tilde{\mathcal{G}}_1$, where the reaction $X_4 \rightarrow X_1$ has a larger rate constant as compared to the rate constant of the same reaction in $\mathcal{G}_1$. The DSR graph corresponding to the network $\tilde{\mathcal{G}}_1$ is given in Figure~\ref{fig:DSR_graph_1}
\begin{figure}[h!]
\centering
\includegraphics[scale=0.32]{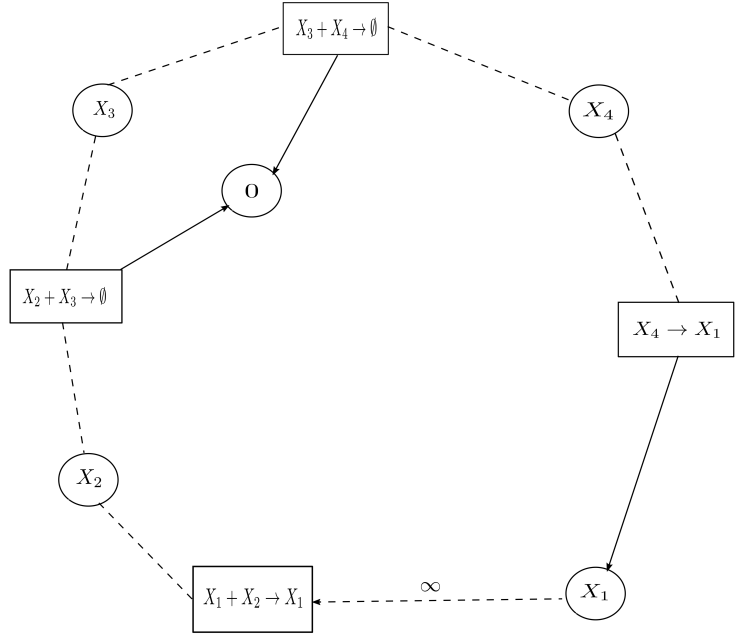}
\caption{DSR graph for network $\tilde{G}_1$.}
\label{fig:DSR_graph_1}
\end{figure}
This DSR graph possesses exactly one oriented cycle given by $(X_3)-(X_3 + X_4\rightarrow \emptyset)-(X_4)-(X_4 \rightarrow X_1)-(X_1)-(X_1 + X_2 \rightarrow X_1)-(X_2)-(X_2 + X_3\rightarrow \emptyset) - (X_3)$ which is an o-cycle. As a consequence, conditions~\ref{s_cycle} and~\ref{e_cycle} in Theorem~\ref{thm:DSR_criterion} are satisfied. In addition, if we choose four reactions from $\tilde{\mathcal{G}}_1$ given by $\{X_1 + X_2 \rightarrow X_1,X_2 + X_3 \rightarrow \emptyset, X_3 + X_4 \rightarrow  \emptyset, X_4 \rightarrow X_1\}$, then we have $det(\y_1,\y_2,\y_3,\y_4)\cdot det(\y'_1-\y_1,\y'_2-\y_2,\y'_3-\y_3,y'_4-\y_4) = det \begin{pmatrix} 
1  &  0  &  0  & 0 \\
1  &  1  &  0  & 0 \\
0  &  1  &  1  & 0 \\
0  &  0  &  1  & 1 \\
\end{pmatrix}\cdot
det\begin{pmatrix} 
0  &   0  &  0  & 1 \\
-1 &  -1  &  0  & 0 \\
0  &  -1  & -1  & 0 \\
0  &   0  & -1  & -1 \\
\end{pmatrix}=1\neq 0$. Therefore, condition~\ref{choice_reactions} of Theorem~\ref{thm:DSR_criterion} is satisfied. Using Theorem~\ref{thm:reaction_network_homeostasis}, we get that $\mathcal{G}_1$ cannot exhibit infinitesimal homeostasis for any choices of rate constants.

\subsection{A reaction network that does exhibit infinitesimal homeostasis}

\textbf{Example:} Consider the following network $\GG_2$

\begin{eqnarray}
\begin{split}
2X_1 &\rightleftharpoons X_2 \\
X_2 + X_3 &\rightarrow X_2 \\
X_1 + X_3 &\rightarrow X_1 + 2X_3 \\
2X_3 &\rightleftharpoons X_3 \\
\emptyset &\xrightleftharpoons[]{\zeta} X_1 \\
\end{split}
\end{eqnarray}
The network $\GG_2$ does not have all the inflow/outflow reactions, but the stoichiometric subspace is full. Using the procedure given in Theorem~\ref{thm:reaction_network_homeostasis}, the homeostasis-associated reaction network $\tilde{\GG}_2$ is given by the following:

\begin{eqnarray}
\begin{split}
2X_1 &\rightarrow 2X_1 + X_2 \\
X_2 + X_3 &\rightarrow X_2 \\
X_1 + X_3 &\rightarrow X_1 + 2X_3 \\
2X_3 &\rightleftharpoons X_3 \\
X_3 &\rightarrow X_1 \\
X_2 &\rightarrow \emptyset 
\end{split}
\end{eqnarray}

\begin{figure}[h!]
\centering
\includegraphics[scale=0.36]{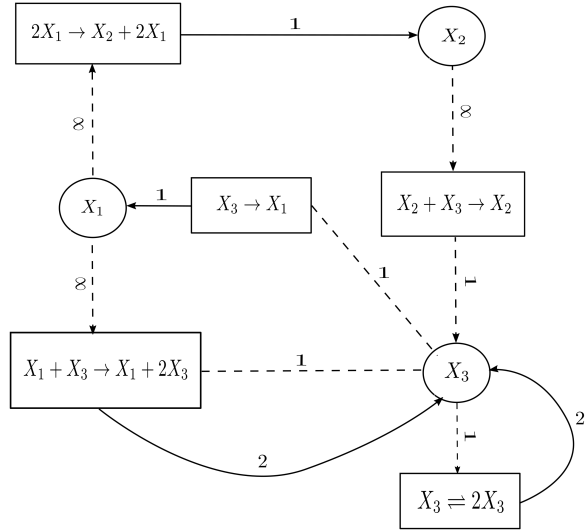}
\caption{DSR graph for network $\tilde{G}_2$.}
\label{fig:DSR_graph_2}
\end{figure}

Let us analyze the DSR graph corresponding to the network $\tilde{\GG_2}$ as shown in Figure~\ref{fig:DSR_graph_2}. Since condition is not satisfied for the DSR graph of $\tilde{\GG_2}$, there is a possibility that the network $\GG_2$ can exhibit infinitesimal homeostasis. In particular, $\GG_2$ generates a dynamical system given by the following set of differential equations:
\begin{eqnarray}\label{eq:dynamical_system_second_example}
\begin{split}
\dot{x_1} &= \zeta - x_1 -2x_1^2 + 2x_2 \\
\dot{x_2} &= x^2_1 - x_2 \\
\dot{x_3} &= -x_2x_3 + x_1x_3 - x^2_3 + x_3 
\end{split}
\end{eqnarray}
This set of differential equations has the steady state given by $(x^*_1 = \zeta, x^*_2 = \zeta^2, x^*_3 = 1 - \zeta + \zeta^2)$. The Jacobian corresponding to~\ref{eq:dynamical_system_second_example} is given by
$$J_2=
\begin{pmatrix} 
-1 - 4x_1 &  2   & 0 \\
2x_1      & -1   & 0 \\
x_3       & -x_3 & -x_2 + x_1 -2x_3 +1 
\end{pmatrix}
$$
The determinant of the $(n-1)\times (n-1)$ minor of  $J_2 $ obtained by deleting its first row and last column is given by $x_3 - 2x_1x_3$ which is $0$ at $x_1=\frac{1}{2}$. We now check the stability of the equilibrium point given by $(x^*_1 = \frac{1}{2}, x^*_2 = \frac{1}{4}, x^*_3 = \frac{3}{4}, \zeta = \frac{1}{2})$. The Jacobian at this point is given by 
$$J^*_2=
\begin{pmatrix} 
-3 &  2   & 0 \\
1  & -1   & 0 \\
\frac{3}{4}  & -\frac{3}{4} & -\frac{1}{4} 
\end{pmatrix}
$$
The Jacobian $J^*_2$ has eigenvalues given by $\lambda_1 = \frac{-1}{4}, \lambda_2 = -\sqrt{3} - 2, \lambda_3 = -\sqrt{3} + 2$, which are all negative and hence the equilibrium is linearly stable. Therefore the network $\GG_2$ exhibits infinitesimal homeostasis at $(x^*_1 = \frac{1}{2}, x^*_2 = \frac{1}{4}, x^*_3 = \frac{3}{4}, \zeta = \frac{1}{2})$.

\subsection{A reaction network that exhibits perfect homeostasis}

\textbf{Example:} Consider the following network $\GG_3$

\begin{eqnarray}
\begin{split}
X_3 + X_1  &\rightarrow X_2 \\
X_2       &\rightarrow X_3 \\
\emptyset &\xrightleftharpoons[]{\zeta} X_1 \\
\emptyset &\xrightleftharpoons[]{} X_3	 
\end{split}
\end{eqnarray}
The network $\GG_3$ generates a dynamical system given by
\begin{eqnarray}\label{eq:dynamical_system_third_example}
\begin{split}
\dot{x_1} &= \zeta - x_1 -x_1x_3 \\
\dot{x_2} &=  - x_2 + x_1x_3 \\
\dot{x_3} &= 1 - x_3 + x_2- x_1x_3
\end{split}
\end{eqnarray}
The Jacobian corresponding to Equation~\ref{eq:dynamical_system_third_example} is given by
$$J(x_1,x_2,x_3)=
\begin{pmatrix} 
-1-x_3 &  0   & -x_1 \\
x_3  & -1   & x_1 \\
-x_3  & 1 & -1-x_1 
\end{pmatrix}
$$
The steady-state corresponding to Equation~\ref{eq:dynamical_system_third_example} is given by $(x^*_1,x^*_2,x^*_3)=\left(\frac{\zeta}{2}, \frac{\zeta}{2}, 1\right)$. Given this steady-state parametrization, one can show that the Jacobian with $\zeta=1$ has all negative eigenvalues given by $\lambda_1 = -1,\lambda_2 = \frac{-7+\sqrt{17}}{4},\lambda_3 = \frac{-7-\sqrt{17}}{4}$. Therefore the point $(\frac{1}{2},\frac{1}{2},1)$ is linearly stable. \\

\noindent The homostasis-associated network $\hat{\GG}_3$ is given by the following:

\begin{eqnarray}
\begin{split}
X_3 + X_1  &\rightarrow X_1 + X_2 \\
X_2       &\rightarrow X_3 \\
X_3       &\rightarrow X_1 \\
\emptyset &\rightleftharpoons X_3 \\
\end{split}
\end{eqnarray}
Every term of the form $det(y_1,y_2,y_3)det(y'_1-y_1,y'_2-y_2,y'_3-y_3)$ from this network is zero and hence the deteminant of the Jacobian is zero. This implies that the determinant of $B$ (which is the $(n-1)\times (n-1)$ minor of the Jacobian obtained by deleting its first row and last column) is zero. By Theorem~\ref{thm:reaction_network_perfect_homeostasis}, we get that the network $\GG_3$ has perfect homeostasis at this linearly stable equilibrium.

\section{Discussion}



In this paper we have analyzed the notion of infinitesimal homeostasis (as introduced in~\cite{golubitsky2017homeostasis}), from the point of view of reaction network models. In particular, we have described a  relationship between infinitesimal homeostasis and network injectivity, as well as a relationship between perfect homeostasis and the structure of the set of reaction vectors. Moreover, since injectivity of a network can be studied by looking at its directed species reaction graph (DSR graph)~\cite{banaji2009graph}, we have discussed how the  DSR graph can be used to analyze homeostasis. 


An interesting direction for future work would be the analysis of possible relationships between  homeostasis (and especially {\em perfect homeostasis}) and \emph{absolute concentration robustness} (ACR). The notion of ACR was first introduced in~\cite{shinar2010structural}, and refers to systems where the value of one of the variables (e.g., species concentration) is the same for all positive steady states of the system. At first, these two notions seem almost identical, but the ACR framework does {\em not} allow for any changes in parameter values. 
A deeper exploration of the mathematical relationships between homeostasis and absolute concentration robustness may uncover other network-level conditions for homeostasis.

Another promising direction for future work is the use of various forms of steady state parametrizations~\cite{perez2018structure,feliu2013variable} to analyze infinitesimal homeostasis. Given a certain steady state parametrization, the fact that the derivative of the output variable with respect to an input variable is zero at an equilibrium manifests itself as a property of a system of algebraic equations, whose analysis could provide useful insights into the behaviour of the system. Possible candidates for this work include toric~\cite{millan2012chemical} and rational~\cite{thomson2009rational} steady state parametrizations.

\section{Acknowledgements}
G.C would like to acknowledge support from NSF grant DMS-1816238 and from a Simons Foundation fellowship. A.D. would like thank the Mathematics Department at University of Wisconsin-Madison for a Van Vleck Visiting Assistant Professorship.

\bibliographystyle{amsplain}
\bibliography{Bibliography}

\providecommand{\bysame}{\leavevmode\hbox to3em{\hrulefill}\thinspace}
\providecommand{\MR}{\relax\ifhmode\unskip\space\fi MR }
\providecommand{\MRhref}[2]{%
  \href{http://www.ams.org/mathscinet-getitem?mr=#1}{#2}
}
\providecommand{\href}[2]{#2}
\begin{thebibliography}{10}

\bibitem{banaji2009graph}
M.~Banaji and G.~Craciun, \emph{Graph-theoretic approaches to injectivity and
  multiple equilibria in systems of interacting elements}, Commun. Math. Sci.
  \textbf{7} (2009), no.~4, 867--900.

\bibitem{bernard1898introduction}
C.~Bernard, \emph{Introduction {\`a} l'{\'e}tude de la m{\'e}decine
  exp{\'e}rimentale}, Librairie Joseph Gilbert, 1898.

\bibitem{cannon1926physiological}
W.~Cannon, \emph{Physiological regulation of normal states: some tentative
  postulates concerning biological homeostatics}, Ses Amis, ses Colleges, ses
  Eleves (1926).

\bibitem{craciun2015toric}
G.~Craciun, \emph{Toric differential inclusions and a proof of the global
  attractor conjecture}, arXiv preprint arXiv:1501.02860 (2015).

\bibitem{craciun2019polynomial}
\bysame, \emph{Polynomial dynamical systems, reaction networks, and toric
  differential inclusions}, SIAGA \textbf{3} (2019), no.~1, 87--106.

\bibitem{craciun2005multiple}
G.~Craciun and M.~Feinberg, \emph{Multiple equilibria in complex chemical
  reaction networks: {I}. the injectivity property}, SIAM J. Appl. Math.
  \textbf{65} (2005), no.~5, 1526--1546.

\bibitem{feinberg1979lectures}
M.~Feinberg, \emph{Lectures on chemical reaction networks}, Notes of lectures
  given at the Mathematics Research Center, University of Wisconsin (1979), 49.

\bibitem{feinberg2019foundations}
\bysame, \emph{Foundations of chemical reaction network theory}, Springer,
  2019.

\bibitem{feliu2013variable}
E.~Feliu and C.~Wiuf, \emph{Variable elimination in post-translational
  modification reaction networks with mass-action kinetics}, J. Math. Biol.
  \textbf{66} (2013), no.~1, 281--310.

\bibitem{felix2016analyzing}
B.~F{\'e}lix, A.~Shiu, and Z.~Woodstock, \emph{Analyzing multistationarity in
  chemical reaction networks using the determinant optimization method}, Appl.
  Math. Comput. \textbf{287} (2016), 60--73.

\bibitem{golubitsky2018homeostasis}
M.~Golubitsky and I.~Stewart, \emph{Homeostasis with multiple inputs}, SIAM J.
  Appl. Dyn. Sys. \textbf{17} (2018), no.~2, 1816--1832.

\bibitem{golubitsky2020input}
M.~Golubitsky, I.~Stewart, F.~Antoneli, Z.~Huang, and Y.~Wang,
  \emph{Input-output networks, singularity theory, and homeostasis},
  Proceedings of the Workshop on Dynamics, Optimization and Computation held in
  honor of the 60th birthday of Michael Dellnitz, Springer, 2020, pp.~31--65.

\bibitem{golubitsky2017homeostasis}
Martin Golubitsky and Ian Stewart, \emph{Homeostasis, singularities, and
  networks}, Journal of mathematical biology \textbf{74} (2017), no.~1-2,
  387--407.

\bibitem{guldberg1864studies}
C.~Guldberg and P.~Waage, \emph{Studies {C}oncerning {A}ffinity}, CM
  Forhandlinger: Videnskabs-Selskabet {I} Christiana \textbf{35} (1864),
  no.~1864, 1864.

\bibitem{gunawardena2003chemical}
J.~Gunawardena, \emph{Chemical reaction network theory for in-silico
  biologists}, Notes available for download at http://vcp. med. harvard.
  edu/papers/crnt. pdf (2003).

\bibitem{joshi2015survey}
B.~Joshi and A.~Shiu, \emph{A survey of methods for deciding whether a reaction
  network is multistationary}, Math. Model. Nat. Phenom. \textbf{10} (2015),
  no.~5, 47--67.

\bibitem{millan2012chemical}
Mercedes~P{\'e}rez Mill{\'a}n, Alicia Dickenstein, Anne Shiu, and Carsten
  Conradi, \emph{Chemical reaction systems with toric steady states}, Bull.
  Math. Biol. \textbf{74} (2012), no.~5, 1027--1065.

\bibitem{nijhout2014homeostasis}
F.~Nijhout and M.~Reed, \emph{Homeostasis and dynamic stability of the
  phenotype link robustness and plasticity}, Integr. Comp. Biol. \textbf{54}
  (2014), no.~2, 264--275.

\bibitem{nijhout2004mathematical}
F.~Nijhout, M.~Reed, P.~Budu, and C.~Ulrich, \emph{A mathematical model of the
  folate cycle new insights into folate homeostasis}, J. Biol. Chem.
  \textbf{279} (2004), no.~53, 55008--55016.

\bibitem{nijhout2014escape}
H.~Nijhout, J.~Best, and M.~Reed, \emph{Escape from homeostasis}, Math. Biosci.
  \textbf{257} (2014), 104--110.

\bibitem{perez2018structure}
M.~Perez~Millan and A.~Dickenstein, \emph{The structure of messi biological
  systems}, SIAM J. Appl. Dyn. Sys. \textbf{17} (2018), no.~2, 1650--1682.

\bibitem{reed2017analysis}
M.~Reed, J.~Best, M.~Golubitsky, I.~Stewart, and H.~Nijhout, \emph{Analysis of
  homeostatic mechanisms in biochemical networks}, Bull. Math. Biol.
  \textbf{79} (2017), no.~11, 2534--2557.

\bibitem{shinar2010structural}
G.~Shinar and M.~Feinberg, \emph{Structural sources of robustness in
  biochemical reaction networks}, Science \textbf{327} (2010), no.~5971,
  1389--1391.

\bibitem{siegel2008linearization}
D.~Siegel and M.~Johnston, \emph{Linearization of complex balanced chemical
  reaction systems}, Preprint (2008).

\bibitem{tang2016design}
Z.~Tang and D.~McMillen, \emph{Design principles for the analysis and
  construction of robustly homeostatic biological networks}, J. Theor. Biol.
  \textbf{408} (2016), 274--289.

\bibitem{thomson2009rational}
M.~Thomson and J.~Gunawardena, \emph{The rational parameterisation theorem for
  multisite post-translational modification systems}, J. Theor. Biol.
  \textbf{261} (2009), no.~4, 626--636.

\bibitem{voit2015150}
E.~Voit, H.~Martens, and S.~Omholt, \emph{150 years of the mass action law},
  PLOS Comput. Biol. \textbf{11} (2015), no.~1, e1004012.

\bibitem{wang2020structure}
Y.~Wang, Z.~Huang, F.~Antoneli, and M.~Golubitsky, \emph{The structure of
  infinitesimal homeostasis in input-output networks}, arXiv preprint
  arXiv:2007.05348 (2020).

\bibitem{yu2018mathematical}
P.~Yu and G.~Craciun, \emph{Mathematical {A}nalysis of {C}hemical {R}eaction
  {S}ystems}, Isr. J. Chem. \textbf{58} (2018), no.~6-7, 733--741.

\end{thebibliography}

\end{document}